\documentclass[12pt,twoside]{amsart}
\usepackage{amsmath, amsthm, amscd, amsfonts, amssymb, graphicx}
\usepackage{enumerate}
\usepackage[colorlinks=true,
linkcolor=blue,
urlcolor=cyan,
citecolor=red]{hyperref}
\usepackage{mathrsfs}
\newtheorem{theorem}{Theorem}[section]

\newtheorem{remark}{Remark}[section]

\newtheorem{corollary}{Corollary}[section]

\newtheorem{proposition}{Proposition}[section]
\numberwithin{equation}{section}

\begin{document}
\title{On further refinements for Young inequalities}
\author{Shigeru Furuichi and Hamid Reza Moradi}
 \maketitle
\begin{abstract}
In this paper, sharp results on operator Young's inequality are obtained. We first obtain sharp multiplicative refinements and reverses for the operator Young's inequality. Secondly, we give an additive result, which improves a well-known inequality due to Tominaga. We also provide some estimates for $A{{\sharp}_{v}}B-A{{\nabla }_{v}}B$ in which $v\notin \left[ 0,1 \right]$.
\end{abstract}
\pagestyle{myheadings}
\markboth{\centerline {S. Furuichi \& H.R. Moradi}}
{\centerline {On inequalities of Young type}}
\bigskip
\bigskip
\section{Introduction}
This note lies in the scope of operator inequalities. We assume that the reader is familiar with the continuous functional calculus and Kubo-Ando theory \cite{kubo}.

It is to be understood throughout the paper that the capital letters present bounded
linear operators acting on a Hilbert space $\mathcal{H}$. $A$ is positive (written $A\ge 0$) in case
$\left\langle Ax,x \right\rangle \ge 0$ for all $x\in \mathcal{H}$ also an operator $A$ is said to be strictly positive(denoted by $A>0$) if $A$ is positive and invertible. If $A$ and $B$ are self-adjoint, we write $B\ge A$ in case $B-A\ge 0$. As usual, by $I$
we denote the identity operator.

The weighted arithmetic mean ${{\nabla }_{v}}$, geometric mean ${{\sharp}_{v}}$, and harmonic mean $!_v$, for $v\in \left[ 0,1 \right]$ and $a,b>0$, are defined as follows:
\[a{{\nabla }_{v}}b=\left( 1-v \right)a+vb,\quad a{{\sharp}_{v}}b={{a}^{1-v}}{{b}^{v}},\quad a!_vb =\left\{(1-v)a^{-1}+vb^{-1}\right\}^{-1}.\] 
If $v=\frac{1}{2}$,  we denote the arithmetic, geometric, and harmonic means, respectively, by $\nabla $, $\sharp$ and $!$, for the simplicity. Like the scalar cases, the operator arithmetic mean, the operator geometric mean, and the operator harmonic mean for $A,B>0$ are defined as follows:
\[A{{\nabla }_{v}}B=\left( 1-v \right)A+vB, \quad A{{\sharp}_{v}}B={{A}^{\frac{1}{2}}}{{\left( {{A}^{-\frac{1}{2}}}B{{A}^{-\frac{1}{2}}} \right)}^{v}}{{A}^{\frac{1}{2}}},\quad A!_vB=\left\{(1-v)A^{-1}+vB^{-1}\right\}^{-1}.\]
The celebrated arithmetic-geometric-harmonic-mean inequalities for scalars assert that if $a,b>0$, then
\begin{equation}\label{8}
a!_v b \leq a{{\sharp}_{v}}b\le a{{\nabla }_{v}}b.
\end{equation}
Generalization of the inequalities \eqref{8}  to operators can be seen as follows: If $A,B>0$, then
\begin{equation*}
A!_vB \leq A{{\sharp}_{v}}B\le A{{\nabla }_{v}}B.
\end{equation*}
The last inequality above is called the operator Young's inequality. During the past years, several refinements and reverses were given for Young's inequality, see for example \cite{5,6,liao}. 

Zuo et al.  showed in \cite[Theorem 7]{1} that the following inequality holds:
\begin{equation}\label{4}
K{{\left( h,2 \right)}^{r}}A{{\sharp}_{v}}B\le A{{\nabla }_{v}}B,\quad\text{ }r=\min \left\{ v,1-v \right\},\text{ }K\left( h,2 \right)=\frac{{{\left( h+1 \right)}^{2}}}{4h},\text{ }h=\frac{M}{m}
\end{equation}
whenever $0<m'I\le B\le mI< MI\le A\le M'I$ or $0<m'I\le A\le mI< MI\le B\le M'I$.
As the authors mentioned in \cite{1}, the inequality \eqref{4} improves the following refinement of Young's inequality involving Specht's ratio $S\left( t \right)=\frac{{{t}^{\frac{1}{t-1}}}}{e\log {{t}^{\frac{1}{t-1}}}}\left( t>0,t\ne 1 \right)$ (see \cite[Theorem 2]{2}), 
\[S\left( {{h}^{r}} \right)A{{\sharp}_{v}}B\le A{{\nabla }_{v}}B.\] 
Another improvement of Young's inequality,  is shown in \cite[Corollary 1]{dragomir}:
\[A{{\nabla }_{v}}B\le \exp \left[ \frac{v\left( 1-v \right)}{2}{{\left( h-1 \right)}^{2}} \right]A{{\sharp}_{v}}B.\]
We remark that there is no relationship between the constants $K{{\left( h,2 \right)}^{r}}$ and $\exp \left[ \frac{v\left( 1-v \right)}{2}{{\left( h-1 \right)}^{2}} \right]$ in general.

In \cite{3,6} we proved some sharp multiplicative reverses of Young's inequality. In this brief note, as the continuation of our previous works, we establish sharp bounds for the arithmetic, geometric and harmonic mean inequalities. Moreover, we shall show some additive-type refinements and reverses of Young's inequality. We will formulate our new results in a more general setting, namely the sandwich assumption $sA\leq B \leq tA$ ($0<s\le t$). Additionally, we provide some estimates for $A{{\sharp}_{v}}B-A{{\nabla }_{v}}B$ in which $v\notin \left[ 0,1 \right]$.
\section{Main Results}
In our previous work \cite{3}, we gave new sharp inequalities for reverse Young inequalities. In this section, we firstly give new sharp inequalities for Young inequalities, as limited cases in the first inequalities both (i) and (ii) of the following theorem. 
\begin{theorem}\label{1}
	Let $A,B>0$ such that $sA\leq B \leq tA$ for some scalars $0<s \leq t$ and let $f_v(x)\equiv \frac{(1-v)+vx}{x^v}$ for $x>0$, and $v \in [0,1]$.
	\begin{itemize}
		\item[(i)] If $t \leq 1$, then $f_v(t) A\sharp_vB \leq A\nabla_v B \leq f_v(s) A\sharp_vB$.
		\item[(ii)] If $s \geq 1$, then $f_v(s) A\sharp_vB \leq A\nabla_v B \leq f_v(t)A\sharp_vB$.
	\end{itemize}
\end{theorem}
\begin{proof}
	Since $f'_v(x)=v(1-v)(x-1)x^{-v-1}$, $f_v(x)$ is monotone decreasing for $0<x \leq 1$ and monotone increasing for $x \geq 1$.
	
	(i) For the case $0<s \leq x \leq t \leq 1$, we have $f_v(t) \leq f_v(x)\leq f_v(s)$, which implies $f_v(t) A\sharp_vB \leq A\nabla_v B \leq f_v(s) A\sharp_vB $ by the standard functional calculus.
	
	(ii) For the case $1\leq s \leq x \leq t$, we have $f_v(s) \leq f_v(x)\leq f_v(t)$
	which implies $f_v(s) A\sharp_vB \leq A\nabla_v B \leq f_v(t)A\sharp_vB$ by the standard functional calculus.
\end{proof}
\begin{remark}
It is worth emphasizing that each assertion in Theorem \ref{1}, implies the other one. For instance, assume that the assertion (ii) holds, i.e.,
\begin{equation}\label{9}
{{f}_{v}}\left( s \right)\le {{f}_{v}}\left( x \right)\le {{f}_{v}}\left( t \right),\quad\text{ }1\le s\le x\le t.
\end{equation}
Let  $t\le 1$, then $1\le \frac{1}{t}\le \frac{1}{x}\le \frac{1}{s}$. Hence \eqref{9} ensures that
\[{{f}_{v}}\left( \frac{1}{t} \right)\le {{f}_{v}}\left( \frac{1}{x} \right)\le {{f}_{v}}\left( \frac{1}{s} \right).\] 
So 
\[\frac{\left( 1-v \right)t+v}{{{t}^{1-v}}}\le \frac{\left( 1-v \right)x+v}{{{x}^{1-v}}}\le \frac{\left( 1-v \right)s+v}{{{s}^{1-v}}}.\] 
Now, by replacing $v$ by $1-v$ we get
\[\frac{\left( 1-v \right)+vt}{{{t}^{v}}}\le \frac{\left( 1-v \right)+vx}{{{x}^{v}}}\le \frac{\left( 1-v \right)+vs}{{{s}^{v}}}\] 
which means
\[{{f}_{v}}\left( t \right)\le {{f}_{v}}\left( x \right)\le {{f}_{v}}\left( s \right),\quad\text{ }0<s\le x\le t\le 1.\]
\end{remark}
\begin{corollary}\label{cor01}
Let $A, B >0$, $m, m', M, M'>0$, and $v\in \left[ 0,1 \right]$.  
\begin{itemize}
	\item[(i)] If $0<m'I\le A\le mI< MI\le B\le M'I$, then 
	\begin{equation}\label{6}
	\frac{m\nabla_v M}{m\sharp_v M}  A\sharp_v B \leq A\nabla_v B\leq \frac{m'\nabla_v M'}{m'\sharp_v M'}A\sharp_v B.	
	\end{equation}
	
	\item[(ii)] If $0<m'I\le B\le mI< MI\le A\le M'I$, then
	\begin{equation}\label{5}
\frac{M\nabla_v m}{M\sharp_v m} A\sharp_v B \leq A\nabla_v B\leq \frac{M'\nabla_v m'}{M'\sharp_v m'}A\sharp_v B.	
	\end{equation}
\end{itemize}
\end{corollary}

\begin{proof}
We use again the function $f_v(x)=\frac{(1-v)+vx}{x^v}$ in this proof.

The condition (i) is equivalent to $I\le \frac{M}{m}I\le {{A}^{-\frac{1}{2}}}B{{A}^{-\frac{1}{2}}}\le \frac{M'}{m'}I$, so that we get $f_v(\frac{M}{m}) A\sharp_v B \leq A\nabla_v B\leq f_v(\frac{M'}{m'})A\sharp_v B$ by putting $s=\frac{M}{m}$ and $t=\frac{M'}{m'}$ in (ii) of Theorem \ref{1}. 

Similarly, the condition (ii) is equivalent to $\frac{m'}{M'}I\le {{A}^{-\frac{1}{2}}}B{{A}^{-\frac{1}{2}}}\le \frac{m}{M}I\le I$, so that we get $f_v(\frac{m}{M}) A\sharp_v B \leq A\nabla_v B\leq f_v(\frac{m'}{M'})A\sharp_v B$ by putting $s=\frac{m'}{M'}$ and $t=\frac{m}{M}$ in (i) of Theorem \ref{1}.

\end{proof}

Note that the second inequalities in both (i) and (ii) of Theorem \ref{1} and Corollary \ref{cor01} are special cases in  Theorem A of our previous paper \cite{3}.

\begin{remark}
	It is remarkable that the inequalities $f_v(t) \leq f_v(x) \leq f_v(s)$ ($0<s\leq x\leq t \leq 1$) given in the proof of Theorem \ref{1} are sharp, since the function $f_v(x)$ for $s\leq x\leq t$ is continuous. So, all result given from Theorem \ref{1} are similarly sharp. As a matter of fact, let $A=MI$ and $B=mI$, then from LHS of \eqref{5}, we infer
\[A{{\nabla }_{v}}B=(M{{\nabla }_{v}}m)I\quad\text{ and }\quad A{{\sharp}_{v}}B=(M{{\sharp}_{v}}m)I.\]
Consequently,
\[\frac{M{{\nabla }_{v}}m}{M{{\sharp}_{v}}m}A{{\sharp}_{v}}B=A{{\nabla }_{v}}B.\]
To see that the constant $\frac{m{{\nabla }_{v}}M}{m{{\sharp}_{v}}M}$ in the LHS of \eqref{6} can not be improved, we consider
 $A=mI$ and $B=MI$, then
\[\frac{m{{\nabla }_{v}}M}{m{{\sharp}_{v}}M}A{{\sharp}_{v}}B=A{{\nabla }_{v}}B.\]
\end{remark}

By replacing $A$, $B$ by ${{A}^{-1}}$, ${{B}^{-1}}$, respectively, then the refinement and reverse of non-commutative geometric-harmonic mean inequality can be obtained as follows:
\begin{corollary}
Let $A, B >0$, $m, m', M, M'>0$, and $v\in \left[ 0,1 \right]$.  
\begin{itemize}
	\item[(i)] If $0<m'I\le A\le mI< MI\le B\le M'I$, then 
	\[\frac{m'{{!}_{v}}M'}{m'{{\sharp}_{v}}M'}A{{\sharp}_{v}}B\le A{{!}_{v}}B\le \frac{m{{!}_{v}}M}{m{{\sharp}_{v}}M}A{{\sharp}_{v}}B.\]
	
	\item[(ii)] If $0<m'I\le B\le mI< MI\le A\le M'I$, then
\[\frac{M'{{!}_{v}}m'}{M'{{\sharp}_{v}}m'}A{{\sharp}_{v}}B\le A{{!}_{v}}B\le \frac{M{{!}_{v}}m}{M{{\sharp}_{v}}m}A{{\sharp}_{v}}B.\]
\end{itemize}
\end{corollary}
Now, we give a new sharp reverse inequality for Young's inequality as an additive-type in the following.
\begin{theorem}\label{3}
Let $A,B>0$ such that $sA\leq B \leq tA$ for some scalars $0<s \leq t$, and $v\in \left[ 0,1 \right]$. Then  
\begin{equation}\label{7}
A{{\nabla }_{v}}B-A{{\sharp}_{v}}B\le \max \left\{ {{g}_{v}}\left( s \right),{{g}_{v}}\left( t \right) \right\}A
\end{equation}
where ${{g}_{v}}\left( x \right)\equiv \left( 1-v \right)+vx-{{x}^{v}}$ for $s \leq x \leq t$.
\end{theorem}
\begin{proof}
Straightforward differentiation shows that $g''_v(x)=v(1-v)x^{v-2} \geq 0$ and $g_v(x)$ is continuous on the interval $\left[ s,t \right]$, so 
\[{{g}_{v}}\left( x \right)\le \max \left\{ {{g}_{v}}\left( s \right),{{g}_{v}}\left( t \right) \right\}.\]
Therefore, by applying similar arguments as in the proof of Theorem \ref{1}, we reach the desired inequality \eqref{7}. This completes the proof of theorem.
\end{proof}

\begin{remark}
We claim that if $A,B>0$ such that $mI\le A,B\le MI$ for some scalars $0<m<M$ and $v\in \left[ 0,1 \right]$, then
\[A{{\nabla }_{v}}B-A{{\sharp}_{v}}B\le \max \left\{ {{g}_{v}}\left( h \right),{{g}_{v}}\left( \frac{1}{h} \right) \right\}A\le L\left( 1,h \right)\log S\left( h \right)A\]
holds, where $L\left( x,y \right)=\frac{y-x}{\log y-\log x}\left( x\neq y \right)$ is the logarithmic mean and the term $S\left( h \right)$ refers to the Specht's ratio. Indeed, we have the inequalities 
$$
(1-v)+vh -h^v \leq L(1,h)\log S(h), \quad (1-v)+v\frac{1}{h}-h^{-v} \leq  L(1,h)\log S(h), 
$$
which were originally proved in \cite[Lemma 3.2]{4}, thanks to $S\left( h \right)=S\left( \frac{1}{h} \right)$ and $L\left( 1,h \right)=L\left( 1,\frac{1}{h} \right)$.
 Therefore, our result, Theorem \ref{3}, improves the well-known result by Tominaga \cite[Theorem 3.1]{4},
\[A{{\nabla }_{v}}B-A{{\sharp}_{v}}B\le L\left( 1,h \right)\log S\left( h \right)A.\]
\end{remark}
\begin{corollary}
	Let $A,B>0$ such that $mI\le A,B\le MI$ for some scalars $0<m < M$. Then
\[A{{\nabla }_{v}}B-A{{\sharp}_{v}}B \le \xi A\]
where $\xi =\max \left\{ \frac{1}{M}\left( M{{\nabla }_{v}}m-M{{\sharp}_{v}}m \right),\frac{1}{m}\left( m{{\nabla }_{v}}M-m{{\sharp}_{v}}M \right) \right\}$.
\end{corollary}

Since $g_v(x)$ is convex so that we can not obtain a general result on the lower bound for $A{{\nabla }_{v}}B-A{{\sharp}_{v}}B$.
However, if we impose the conditions, we can obtain new sharp inequalities for Young inequalities as an additive-type in the first inequalities both (i) and (ii) in the following proposition. (At the same time, of course, we also obtain the upper bounds straightforwardly.)

\begin{proposition}\label{2}
Let $A,B>0$ such that $sA\le B\le tA$ for some scalars $0<s\le t$, $v\in \left[ 0,1 \right]$,  and ${{g}_{v}}$ is defined as in Theorem \ref{3}.
\begin{itemize}
	\item[(i)] If $t\le 1$, then ${{g}_{v}}\left( t \right)A\le A{{\nabla }_{v}}B-A{{\sharp}_{v}}B \le g_v(s)A$.  
	\item[(ii)] If $s\ge 1$, then ${{g}_{v}}\left( s \right)A\le A{{\nabla }_{v}}B-A{{\sharp}_{v}}B \le g_v(t)A$.  
\end{itemize}
\end{proposition}
\begin{proof}
It follows from the fact that  ${{g}_{v}}\left( x \right)$ is monotone decreasing for $0<x\le 1$ and monotone increasing for $x\ge 1$. 
\end{proof}
\begin{corollary}
	Let $A, B >0$, $m, m', M, M'>0$, and $v\in \left[ 0,1 \right]$.  
	\begin{itemize}
		\item[(i)] If $0<m'I\le A\le mI< MI\le B\le M'I$, then
		 $$\frac{1}{m}\left( m{{\nabla }_{v}}M-m{{\sharp}_{v}}M \right)A\le A{{\nabla }_{v}}B-A{{\sharp}_{v}}B\leq \frac{1}{m'}\left( m'{{\nabla }_{v}}M'-m'{{\sharp}_{v}}M' \right)A.$$
		 
		\item[(ii)] If $0<m'I\le B\le mI< MI\le A\le M'I$,  then 
		 $$\frac{1}{M}\left( M{{\nabla }_{v}}m-M{{\sharp}_{v}}m \right)A\le A{{\nabla }_{v}}B-A{{\sharp}_{v}}B\leq \frac{1}{M'}\left( M'{{\nabla }_{v}}m'-M'{{\sharp}_{v}}m' \right)A.$$
	\end{itemize}
\end{corollary}
\begin{remark}
It is known that for any $A,B>0$,
	\[A{{\nabla }_{v}}B\le A{{\sharp}_{v}}B\quad\text{ for }v\notin \left[ 0,1 \right].\]
Assume ${{g}_{v}}\left( x \right)$ is defined as in Theorem \ref{3}. By an elementary computation we have	
\[\left\{ \begin{array}{ll}
 g_{v}^{'}\left( x \right)>0&\text{ for }v\notin \left[ 0,1 \right]\text{ and }0<x\le 1 \\ 
g_{v}^{'}\left( x \right)<0&\text{ for }v\notin \left[ 0,1 \right]\text{ and }x>1 \\ 
\end{array} \right..\]
Now, in the same way as above we have also for any $v\notin \left[ 0,1 \right]$:
\begin{itemize}
	\item[(i)] If $0<m'I\le A\le mI\leq MI\le B\le M'I$, then
\[\frac{1}{m'}\left( m'{{\sharp}_{v}}M'-m'{{\nabla }_{v}}M' \right)A\le A{{\sharp}_{v}}B-A{{\nabla }_{v}}B\le \frac{1}{m}\left( m{{\sharp}_{v}}M-m{{\nabla }_{v}}M \right)A.\]
On account of assumptions, we also infer 
\[(m'{{\sharp}_{v}}M'-m'{{\nabla }_{v}}M')I\le A{{\sharp}_{v}}B-A{{\nabla }_{v}}B\le (m{{\sharp}_{v}}M-m{{\nabla }_{v}}M)I.\]

	\item[(ii)] If $0<m'I\le B\le mI\leq MI\le A\le M'I$, then
\[\frac{1}{M}\left( M{{\sharp}_{v}}m-M{{\nabla }_{v}}m \right)A\le A{{\sharp}_{v}}B-A{{\nabla }_{v}}B\le \frac{1}{M'}\left( M'{{\sharp}_{v}}m'-M'{{\nabla }_{v}}m' \right)A.\]
On account of assumptions, we also infer 
\[(M{{\sharp}_{v}}m-M{{\nabla }_{v}}m)I\le A{{\sharp}_{v}}B-A{{\nabla }_{v}}B\le (M'{{\sharp}_{v}}m'-M'{{\nabla }_{v}}m')I.\]
\end{itemize}
In addition, with the same assumption to Theorem \ref{3} except for $v\notin [0,1]$, we have
$$
\min\{g_v(s),g_v(t)\}A \leq A\nabla_v B-A\sharp_v B,
$$
since we have $\min\{g_v(s),g_v(t)\}\leq g_v(x)$ by $g_v''(x) \leq 0$, for $v \notin[0,1]$.
\end{remark}

\section*{Acknowledgement}
The author (S.F.) was partially supported by JSPS KAKENHI Grant Number 16K05257.

\vskip 0.4 true cm

\tiny(S. Furuichi) Department of Information Science, College of Humanities and Sciences, Nihon University, 3-25-40, Sakurajyousui, Setagaya-ku, Tokyo, 156-8550, Japan.

{\it E-mail address:} furuichi@chs.nihon-u.ac.jp

\vskip 0.4 true cm

\tiny(H.R. Moradi) Young Researchers and Elite Club, Mashhad Branch, Islamic Azad University, Mashhad, Iran

{\it E-mail address:} hrmoradi@mshdiau.ac.ir
\end{document}